\documentclass[a4paper,10pt]{article}
\usepackage[utf8]{inputenc}
\usepackage{amssymb,latexsym}
\usepackage{amsmath,hyperref}
\usepackage{graphicx}
\usepackage{amssymb,amsmath,amsthm,bm}

\newtheorem{theorem}{Theorem}

\newtheorem{claim}{Claim}

\newtheorem{definition}{Definition}
\def\qed{\ifhmode\unskip\nobreak\fi\quad\ifmmode\Box\else$\Box$\fi}

\title{Saturated $2$-plane drawings with few edges}
\author{J\'anos Bar\'at\thanks{Supported by NKFIH grant K-131529 and ERC Advanced Grant “GeoScape” No. 882971.}\\
\small  Department of Mathematics, University of Pannonia and \\
\small Alfr\'ed R\'enyi Institute of Mathematics\\
\small \texttt{barat@mik.uni-pannon.hu}\\
and\\
G\'eza T\'oth\thanks{Supported by NKFIH grant K-131529 and ERC Advanced Grant “GeoScape” No. 882971.}\\
\small Alfr\'ed R\'enyi Institute of Mathematics\\
\small \texttt{toth.geza@renyi.mta.hu}}

\begin{document}

\maketitle
\begin{abstract}
A drawing of a graph is $k$-plane if every edge contains at most $k$
crossings. A $k$-plane drawing is saturated if we cannot add any edge
so that the drawing remains $k$-plane.
It is well-known that saturated $0$-plane drawings, that is,
maximal plane graphs, of $n$ vertices have exactly
$3n{-}6$ edges. For $k>0$, the number of edges of
saturated $n$-vertex $k$-plane graphs can take many different values.
In this note, we establish some bounds on the minimum number of edges
of saturated $2$-plane graphs under various conditions.
\end{abstract}


\section{Preliminaries}

In a  drawing of a graph in the plane,
vertices are represented by points, edges are represented by curves
connecting the points, which correspond to adjacent vertices. 
The points (curves) are also
called vertices (edges).
We assume that an edge does not go through any vertex,
and three edges do not cross at the same point.
A graph together with its drawing is a {\em topological graph}. 
A drawing or a topological graph is {\em simple} if
any two edges have at most one point in common, that is either a common endpoint or a crossing.
In particular, there is no self-crossing.
In this paper, we assume the underlying graph has
neither loops nor multiple edges. 

For any $k\ge 0$, a topological graph is {\it $k$-plane} if each edge contains at most $k$ crossings.
A  graph $G$ is  {\it $k$-planar}
if it has a {\it $k$-plane} 
drawing in the plane.

There are several versions of these concepts, see e.g.
\cite{kozos}. The most studied one is when we consider only simple drawings.
A  graph $G$ is  {\it simple $k$-planar}
if it has a {\it simple $k$-plane} 
drawing in the plane.

A simple $k$-plane drawing is {\em saturated} if no edge can be added
so that the obtained drawing is also simple $k$-plane.
The $0$-planar graphs are the well-known planar graphs.
A plane graph of $n$ vertices has at most $3n-6$ edges. 
If it has exactly  $3n-6$ edges, then it is a triangulation of the plane.
If it has fewer edges, then we can add some edges so that it
becomes a triangulation with $3n-6$ edges.
That is, saturated plane graphs have $3n-6$ edges.

Pach and Tóth \cite{pt} proved the maximum number of edges of an $n$-vertex 
(simple) $1$-planar graph is $4n-8$.
Brandenburg et al. \cite{brandenburg} noticed that
saturated simple $1$-plane graphs can have much fewer edges, namely
$\frac{45}{17}n + O(1)\approx 2.647n$.
Bar\'at and T\'oth \cite{bt} proved that  a
saturated simple $1$-plane graph has at least
$\frac{20n}{9}-O(1)\approx 2.22n$ edges. 

For any $k, n$, let $s_k(n)$ be the minimum number of edges of a saturated $n$-vertex simple $k$-plane drawing.
With these notations,
$\frac{45n}{17}+O(1)\ge s_1(n)\ge \frac{20}{9}n-O(1)$.
For $k>1$, the best bounds known for $s_k(n)$ are shown by
Auer et al \cite{auer} and by Klute and Parada \cite{kp}.
Interestingly for $k\ge 5$ the bounds are very close.

In this note, we concentrate on $2$-planar graphs on $n$ vertices. Pach and Tóth \cite{pt} showed
the maximum number of edges of a (simple) $2$-planar graph is $5n-10$.
Auer et al \cite{auer} and  Klute and Parada \cite{kp} proved that 
$\frac{4n}{3}+O(1)\ge s_2(n)\ge \frac{n}{2}-O(1)$.
We improve the lower bound.

\begin{theorem}\label{n-1}
For any $n>0$, $s_2(n)\ge n-1$.
\end{theorem}

A drawing is {\em $l$-simple} if
any two edges have at most $l$  points in common.
By definition a simple drawing is the same as a $1$-simple drawing.
Let $s_k^l(n)$ be the minimum number of edges of a saturated $n$-vertex $l$-simple $k$-plane drawing.
In \cite{kp} it is shown that
$\frac{4n}{5}+O(1)\ge s_2^2(n)\ge \frac{n}{2}-O(1)$ and
$\frac{2n}{3}+O(1)\ge s_2^3(n)\ge \frac{n}{2}-O(1)$.
We make the following improvements:

\begin{theorem}\label{2-3-simple}
(i)
$s_2^2(3)=3$, and 
$\lfloor3n/4\rfloor \ge s_2^2(n)\ge \lfloor2n/3\rfloor$ for $n\neq 3$,\\
  (ii) $s_2^3(3)=3$, and 
  $s_2^3(n)=\lfloor2n/3\rfloor$ for $n\neq 3$.
\end{theorem}

The saturation problem for $k$-planar graphs has many different settings,
we can allow self-crossings, parallel edges, or we can consider non-extendable {\em abstract} graphs.
See \cite{kozos} for many recent results and a survey.

\section{Proofs}


\begin{definition}
  Let $G$ be a topological graph and $u$ a vertex of degree $1$.
For short, $u$ is called a {\em leaf} of $G$. 
Let $v$ be the only neighbor of $u$.
The pair $(u, uv)$ is called a {\em flag}.
If there is no crossing on $uv$, then $(u, uv)$ is an {\em empty flag}.
\end{definition}

\begin{definition}
Let $G$ be an $l$-simple $2$-plane topological graph.
If an edge contains two crossings, then its piece between the two crossings is a {\em middle segment}.
The edges of $G$ divide the plane into cells.
A cell $C$ is {\em special} if it is bounded only by middle segments and isolated vertices.
Equivalently, $C$ is {\em special},
if there is no vertex on its boundary, apart from isolated vertices. 
An edge that bounds a special cell is also {\em special}.
\end{definition}

Let $G$ be a saturated $l$-simple $2$-plane topological graph, where $1\le l\le 3$.  
Suppose a cell $C$ contains an isolated vertex $v$.
Since $G$ is saturated, $C$ must be a special cell and 
there is no other isolated vertex in
$C$.
Now suppose $C$ is  an empty special cell. 
Each boundary edge contains two crossings. 
Therefore, if we put an isolated vertex in $C$, then the topological graph remains saturated. 
So if we want to prove a lower bound on the number of edges, 
we can assume without loss of generality that each special cell contains an isolated vertex.


\bigskip

\begin{claim} \label{only1}
 A special edge can bound at most one special cell.
\end{claim}

\begin{proof}
Suppose $uv$ is a special edge and let $pq$ be its middle segment.
If $uv$ bounds more than one special cell, then there is a special cell on both sides of $pq$, $C_1$ and $C_2$ say.
Let $p$ be a crossing of the edges $uv$ and $xy$.
There is no crossing on $xy$ between $p$ and one of the
endpoints, $x$ say. 
Therefore, one of the cells $C_1$ and $C_2$ has $x$ on its boundary, a
contradiction.
\end{proof}

\begin{proof}[Proof of {Theorem \ref{n-1}}]
Suppose $G$ is a saturated simple 2-plane
topological graph of $n$ vertices and $e$ edges. We assume that  each special cell contains an isolated vertex.

\begin{claim} \label{emptyflag}
 All flags are empty in $G$.
\end{claim}

\begin{proof}
Let $(u, uv)$ be a flag. 
Suppose to the contrary  there is at least one crossing on $uv$.
Let $p$ be the crossing on $uv$ closest to $u$, with edge $xy$. Since it is a 2-plane drawing, there is no crossing on $xy$ between $p$ and one of the endpoints, $x$ say.
In this case, we can connect $u$ to $x$ along $up$ and $px$. 
Since the drawing was saturated, $u$ and $x$ are adjacent in $G$, and $x\neq v$, that contradicts to $d(u)=1$.
\end{proof}

Remove all empty flags from $G$. 
Observe the resulting topological graph $G'$ is also saturated. 
If we can add an edge to $G'$, then we could have added the same edge to $G$.

Suppose to the contrary that $G'$ contains a flag $(v, vw)$.
Since $G'$ is saturated, the flag is empty by Claim~\ref{emptyflag}.
In $G$, vertex $v$ had degree at least $2$, so $v$ had some other neighbors, 
$u_1, \ldots, u_m$ say, in clockwise order. The flags $(u_i,u_iv)$ were all empty. 
However, $u_1$ can be connected to $w$, which is a contradiction. 
Therefore, there are no flags in $G'$.
On the other hand, the graph $G'$ may contain isolated vertices.
Let $n'$ and $e'$ denote the number of vertices and edges of $G'$.
Since $n-n'=e-e'$, it suffices to show that $e'\ge n'-1$. 
If there are no isolated vertices in $G'$,
then $e'\ge n'$ is immediate. 

We assign weight $1$ to each edge. 
If $G'$ has no edge, then it has one vertex and we are done.
We discharge the weights to the vertices so that each
vertex gets weight at least $1$.
If $uv$ is not a special edge, then 
it gives weight $1/2$ to both endpoints $u$ and $v$.
Suppose now that $uv$ is a special edge. 
It bounds the special cell $C$ containing the isolated vertex $x$. 
If $d(u)=2$, then $uv$ gives weight $1/2$ to $u$, if $d(u)\ge 3$, then 
it gives weight $1/3$ to $u$. We similarly distribute the weight to vertex $v$.
We give the remaining weight of $uv$ to $x$.

We show that each vertex gets weight at least 1. 
This holds immediately for all vertices of positive degree. 
We have to show the statement only for isolated vertices.
Let $x$ be an isolated vertex in a special cell $C$ bounded by
$e_1, e_2, \ldots, e_m$ in clockwise direction.
Let $e_i=u_iv_i$ such that the oriented curve $\overrightarrow{u_iv_i}$ has $C$ on its right. 
See Figure~\ref{specialdisch} for $m=5$. 
Let $p_i$ be the crossing of $e_i$ and $e_{i+1}$. Indices are understood modulo $m$.
In general, it may happen that some of the points in
$\{\ u_i, v_i\ | \ i=1, \ldots, m\ \}$ coincide.
For each vertex $u_i$ or $v_i$ of degree at least $3$, the corresponding boundary edge of $C$ has a remainder charge at least $1/6$.
We have to prove that (with multiplicity) at least $6$ of the vertices 
$u_i$, $v_i$ have degree at least $3$.
Consider vertex $v_i$.

\smallskip

Case 1: $v_i=u_{i+2}$.
The vertex $v_i=u_{i+2}$ can be connected to $u_{i+1}$ along the segments $v_ip_i$ and $p_iu_{i+1}$,
that are crossing-free segments of the corresponding edges.
Similarly, $v_i=u_{i+2}$ can be connected to $v_{i+1}$ along $v_ip_{i+1}$ and $p_{i+1}v_{i+1}$.
Since the drawing was simple and saturated, 
$u_i$, $u_{i+1}$, $v_{i+1}$, $v_{i+2}$ are all different and they are already connected to $v_i=u_{i+2}$,
so it has degree at least $4$.

\smallskip

\begin{figure}
\begin{center}
\scalebox{0.25}{\includegraphics{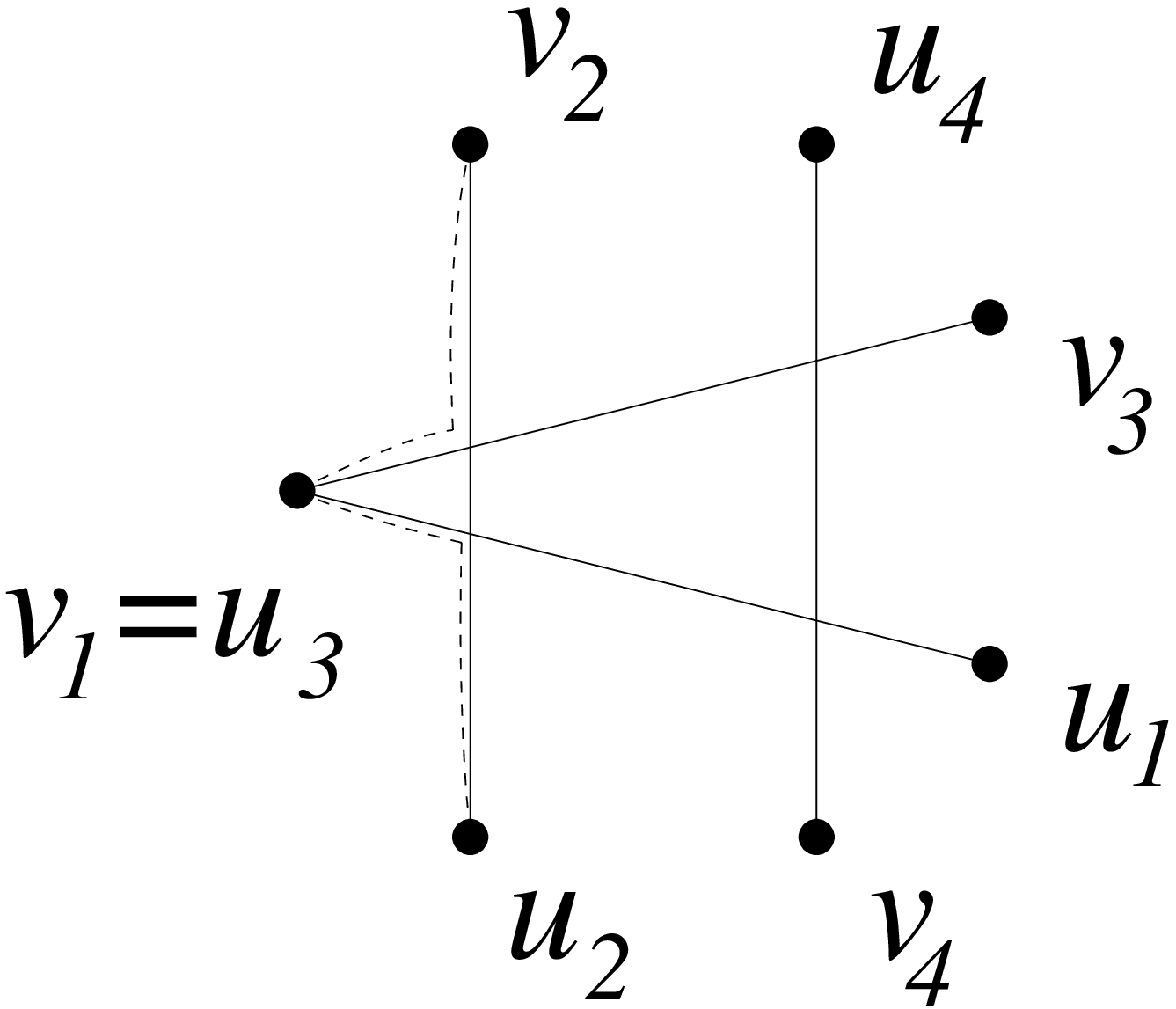}\hspace{3cm} \includegraphics{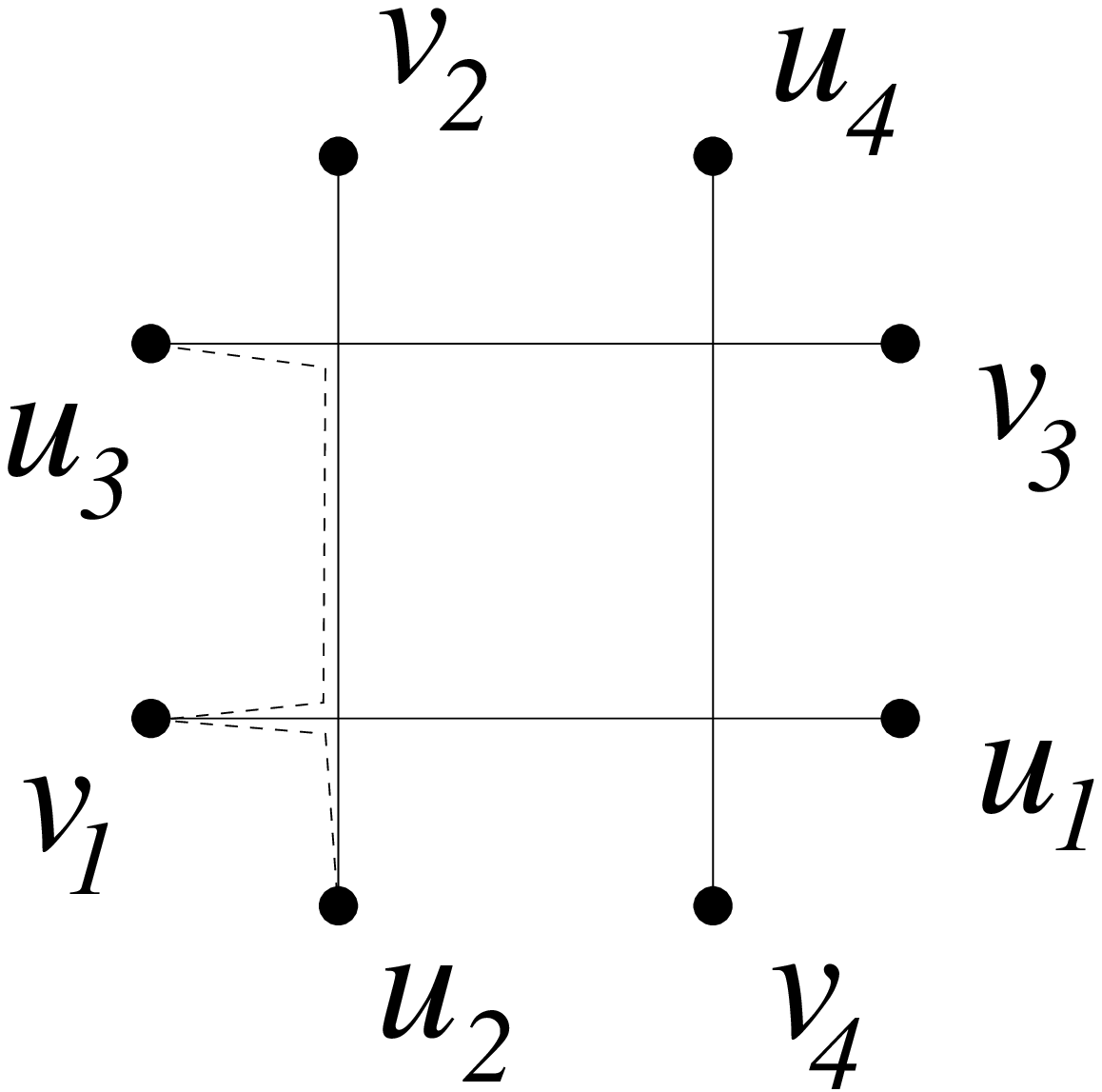}\hspace{3cm}\includegraphics{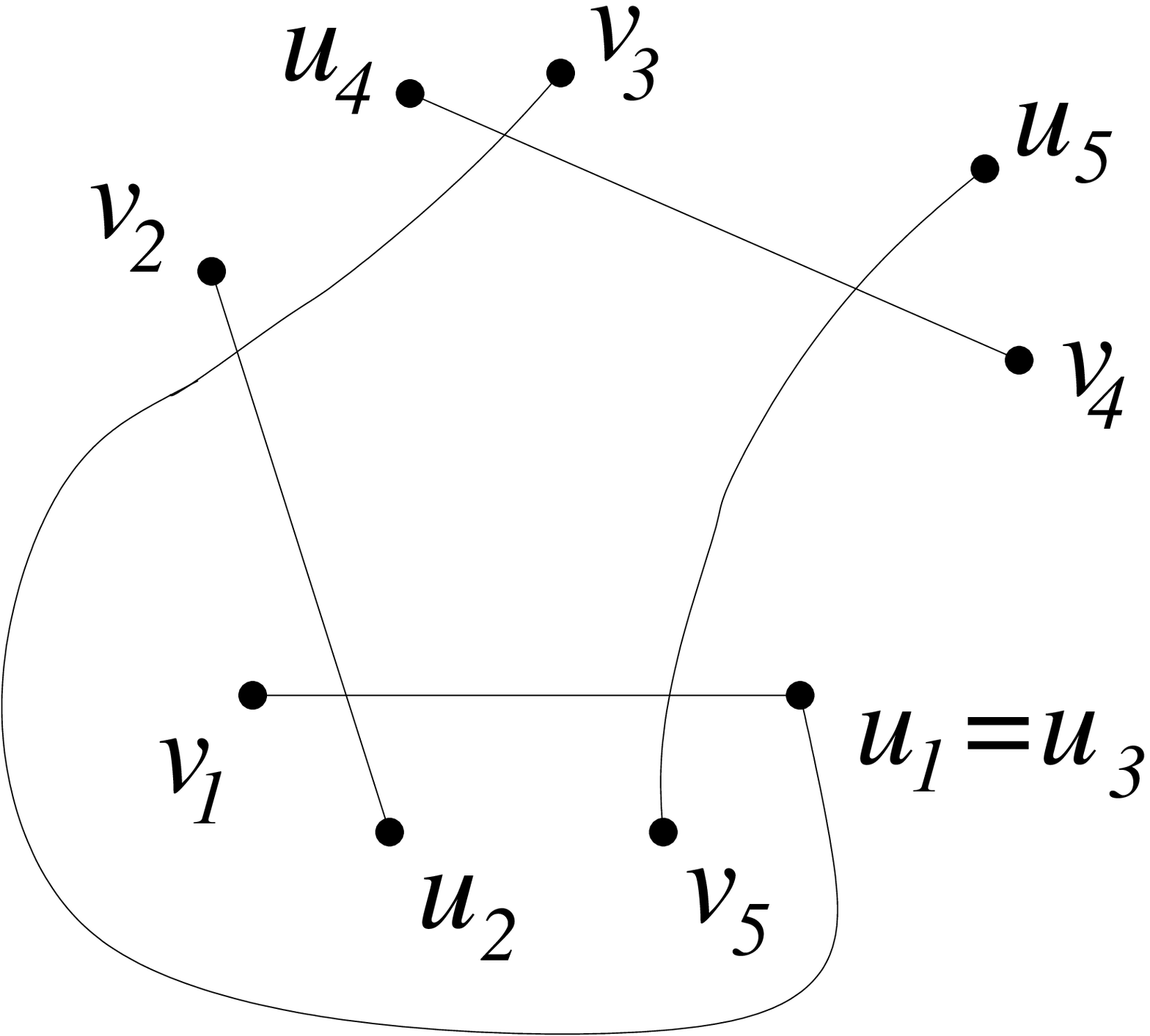} }
\caption{Case 1, $d(v_1)\ge 4$, \hfil Case 2, $d(v_1)\ge 3$ and Case 2, $u_1=u_3$.}\label{specialdisch}
\end{center} 
\end{figure}



Case 2: $v_i\neq u_{i+2}$.
The vertex $v_i$ can be connected to $u_{i+1}$ as before, and to
$u_{i+2}$ along 
$v_ip_i$, $p_ip_{i+1}$ and $p_{i+1}u_{i+2}$.
Since the drawing was saturated, $v_i$ is already adjacent to $u_i$, $u_{i+1}$, $u_{i+2}$.
Unless $u_i=u_{i+2}$, vertex $v_i$ has degree at least 3.
Note that $u_{i+1}\neq u_i$ and $u_{i+1}\neq u_{i+2}$,
since the drawing was $1$-simple.

We can argue analogously for $u_i$.
We conclude that $v_i$ has degree $2$ only if $u_i=u_{i+2}$, and
$u_i$ has degree $2$ only if $v_{i}=v_{i-2}$.

Recall that $m$ is the number of bounding edges of the special cell $C$.
For $m=3$, it is impossible that $u_i=u_{i+2}$ or  $v_{i}=v_{i-2}$, 
therefore, 
for $i=1, 2, 3$ all six vertices $u_i$, $v_i$ have degree at least $3$.

Let $m>3$, and suppose $v_1$ has degree $2$, consequently 
$u_1=u_3$. In this case, we prove that
$u_m$, $u_1$, $u_2$, $u_3$, $v_m$, $v_2$
all have degree at least $3$. 

We show it for $u_2$, the argument is the same for the other vertices.
Let $\gamma$ be the closed curve formed by the segments $u_1p_1$, $p_1p_2$ and $p_2u_3$.
(We have $u_1=u_3$.)
Suppose $d(u_2)=2$. By the previous observations, $v_m=v_2$.
However, $v_m$ and $v_2$ lie on different sides of $\gamma$, therefore they
cannot coincide.
Therefore, there are always at least six vertices $u_i$, $v_i$, with multiplicity,
which
have degree at least $3$, so the isolated vertex $x$ gets weight at least $1$.
This concludes the proof.
\end{proof}

We recall that $s^3_2(n)$ denotes the minimum number
of edges of a saturated $n$-vertex 3-simple 2-plane drawing.


\medskip

\begin{proof}[Proof of {Theorem~\ref{2-3-simple}}]

We start with the upper bounds.
Let $$f(n)=
\left\{ \begin{array}{rl}
3 & \mbox{if} \ \ n=3 \\
\lfloor 3n/4\rfloor & \mbox{otherwise} 
\end{array}\right.
$$


\begin{figure}\label{propel}
\begin{center}
\scalebox{0.25}{\includegraphics{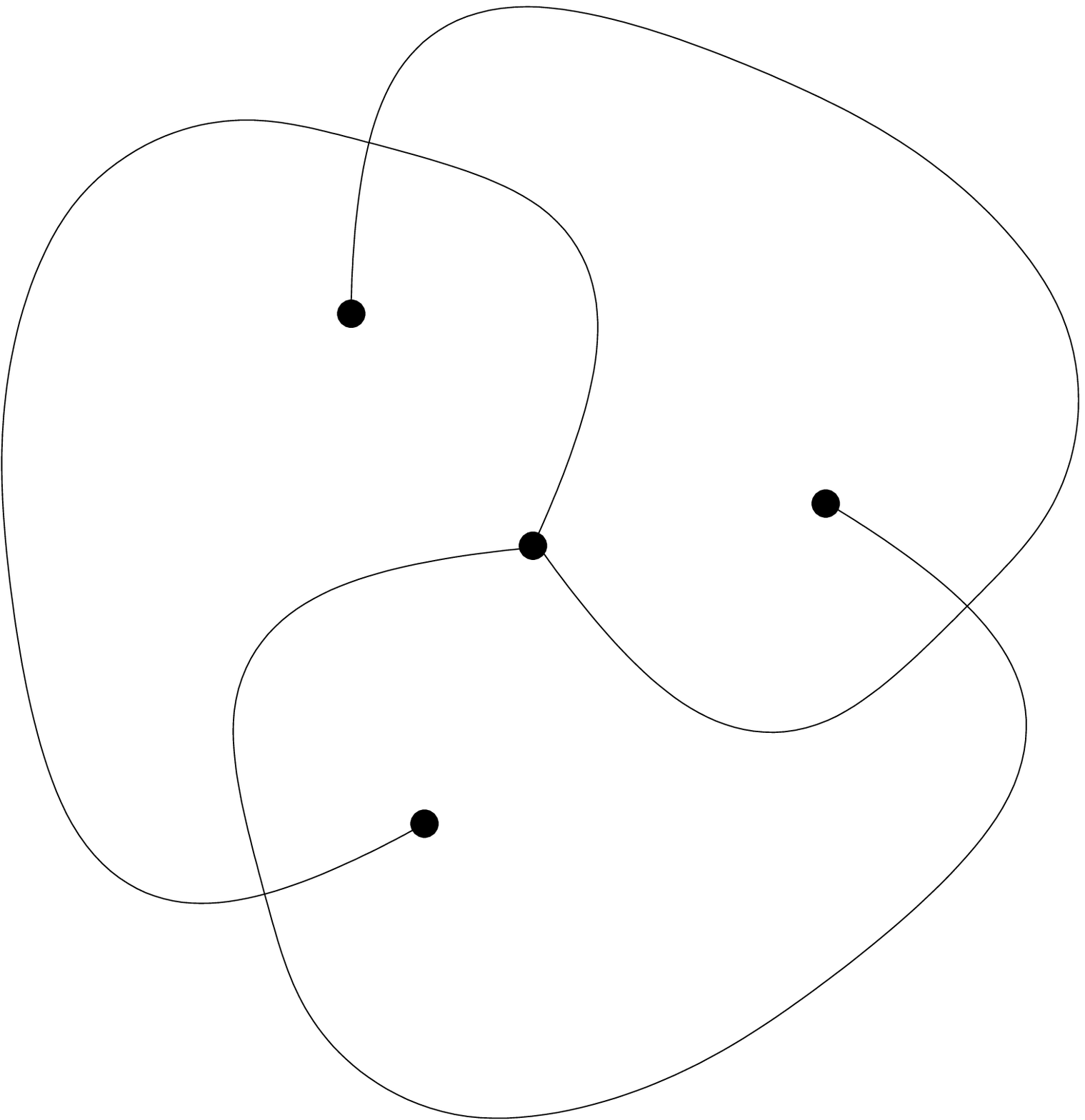}\hspace{3cm} \includegraphics{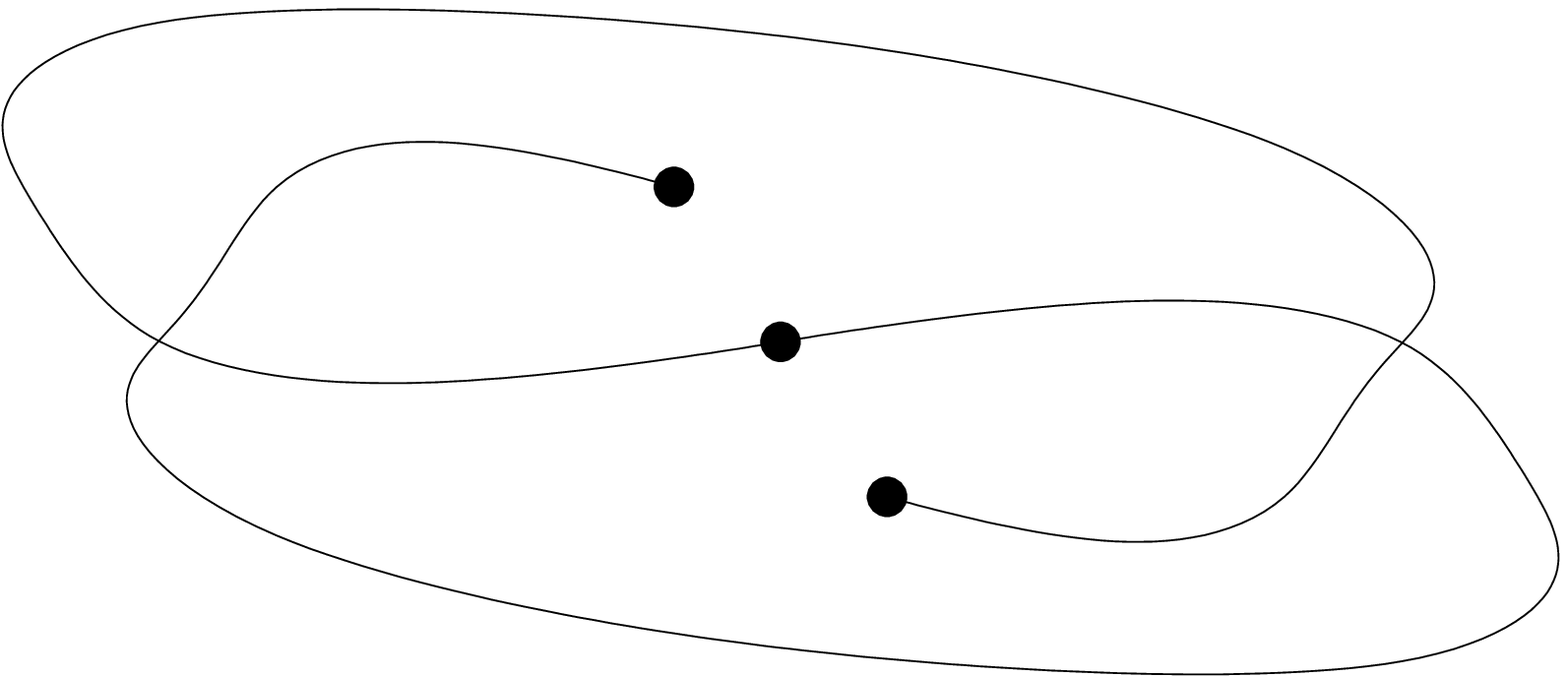}}
\caption{A $3$-propeller and a $2$-propeller.}
\end{center} 
\end{figure}


First we construct a saturated $2$-plane, $2$-simple topological graph
with $n$ vertices and $f(n)$ edges, for every~$n$.
Let $k\ge 3$. A $k$-{\em propeller} is isomorphic to a star with $k$ edges as an abstract graph,
drawn as in  Figure~\ref{propel}. 
Clearly it is a  saturated $2$-plane, $2$-simple topological graph  with $k+1$ vertices,
$k$ edges and the unbounded cell is special.

For $n=1, 2, 3$, a complete graph of $n$ vertices satisfies the statement.
For $n\ge 4$, $n\equiv 0 \bmod 4$, consider $n/4$ disjoint $3$-propellers such that
each of them is in the unbounded cell of the others.
For $n\ge 4$, $n\equiv 1, 2, 3 \bmod 4$, 
replace one of the propellers by an isolated vertex, a $K_2$, and  a $4$-propeller, respectively.
This implies the upper bound in (i), that is,  $s_2^2(n)\le f(n)$.

Now we construct a saturated $2$-plane, $3$-simple topological graph
with $n$ vertices and $\lfloor 2n/3\rfloor$ edges, for every $n$.
A $2$-{\em propeller} is isomorphic to a path of $2$ edges as an abstract graph,
drawn as in  Figure~\ref{propel}. 
Clearly it is a  saturated $2$-plane, $3$-simple topological graph  with $3$ vertices,
$2$ edges and the unbounded cell is special.

For $n\equiv 0 \bmod 3$, take $n/3$ disjoint $2$-propellers such that
each of them is in the unbounded cell of the others.
For $n\equiv 1, 2 \bmod 3$, add an isolated vertex or an independent edge.
This implies the upper bound in (ii), $s_2^3(n)\le \lfloor 2n/3\rfloor$.

\smallskip


We prove by induction on $n$ that  $s_2^2(n)\ge \lfloor2n/3\rfloor$ and
 $s_2^3(n)\ge \lfloor2n/3\rfloor$.
It is trivial for $n\le 4$. Let $n>4$ and assume that 
$s_2^2(m), s_2^3(m)\ge \lfloor2m/3\rfloor$ for every $m<n$.
Let $G$ be a saturated $2$-plane,
$2$-simple or 
$3$-simple drawing with $n$ vertices and $e$ edges.
We may assume again that every special cell contains an isolated vertex.

Suppose that $(u,uv)$ is an empty flag. 
We remove $u$ from $G$. 
Analogous to the proof of Theorem~\ref{n-1},
the obtained topological graph is saturated, it has
$n-1$ vertices and $e-1$ edges. 
By the induction hypothesis, $e-1\ge \lfloor2(n-1)/3\rfloor$,
which implies that $e\ge\lfloor2n/3\rfloor$.
Therefore, we assume for the rest of the proof that
$G$ does not contain empty flags.

\begin{claim} \label{deg3}
If $(u, uv)$ is a flag,
then either $d(v)\ge 3$ or $u$ and $v$ are included in a $2$-propeller.
\end{claim}

\begin{proof}
Since $G$ does not contain empty flags,
there is a crossing on $uv$.
Let $p$ be the crossing on $uv$ closest to $u$, with edge $xy$.
There is no crossing on $xy$ between $p$ and one of the
endpoints, $x$ say, and $x\neq u$ by the assumptions. 
We can connect $u$ to $x$ along the segments $up$ and $px$.
Since the drawing was saturated, $u$ and $x$ are adjacent in $G$.
Since $u$ has degree 1, $x=v$. 
This implies $d(v)\ge 2$. We exclude parallel edges, so $y\neq u$.

Suppose $d(v)=2$.
There is a crossing on the segment $py$ of $vy$, otherwise we could connect
$u$ to $y$ along the segments $up$ and $py$ contradicting the degree assumption on $u$. 
Let $q$ be the crossing of $vy$ and $ab$.
There is no crossing on $ab$ between $q$ and one of the
endpoints, $a$ say.
If $a$ and $u$ are on the same side of edge $vy$ (that is, the directed
edges $\overrightarrow{ab}$ and $\overrightarrow{uv}$ cross the directed edge $\overrightarrow{vy}$ from the same side),
then we can connect $u$ to $a$ along the segments $up$, $pq$, $qa$.
Therefore $a=v$, so either $d(v)\ge 3$, or $b=u$, and edges $uv$ and $vy$ form a $2$-propeller. Note that this case is possible only if $G$ is $3$-simple.

So we may assume that $a$ is on the other side.
If $a=v$, then $d(v)\ge 3$, so we also assume that $a\neq v$. 
Consider now the edge $uv$. 
If there was no crossing on the segment $pv$ of $uv$, then we can connect $u$ to $a$ along $up$,
the segment $pv$ of $yv$, the segment $vp$ of $uv$, $pq$, and $qa$.
Therefore, there is a crossing on the segment $pv$ of $uv$.
Let $r$ be this crossing of $uv$ with edge $cd$,
and we can assume there is no crossing on the segment $cr$. (Here, $c$ or $d$ might coincide with $a$.)
If $c$ and $y$ are on the same side of $uv$ (that is, the directed
edges $\overrightarrow{vy}$ and $\overrightarrow{dc}$ cross the directed edge $\overrightarrow{vu}$ from the same side),
then we can connect $u$ to $c$ along $up$, $px$, $xr$, $rc$, which means that $c=v$, so $d(v)\ge 3$.
If $c$ and $y$ are on opposite sides of $uv$, then 
we can connect $c$ to $v$, so they are already connected.
Therefore, $c=y$. 
However, we assumed that
$\overrightarrow{vy}$ and $\overrightarrow{dc}$ cross the directed edge $\overrightarrow{vu}$ from the opposite sides, so there is another
crossing of $uv$ and $vy$. If $G$ is $2$-simple, this is impossible and we are done. If $G$ is $3$-simple, then this crossing can only be $r$, so $c=y$ and $d=x$. 
Now the edges $uv$ and $vy$ form a $2$-propeller. 
\end{proof}


In a graph $G$, a connected component with at least two vertices is an {\em essential component}. 
If $G$ has only one essential component, then $G$ is {\em essentially connected}.

\begin{claim} \label{ess}
We can assume without loss of generality that $G$ is essentially connected.
\end{claim}


\begin{proof}

Suppose to the contrary $G$ has at least two essential components.
We define a partial order on the essential components of $G$:
$G_i\prec G_j$ if and only if $G_i$
lies in a bounded cell of $G_j$.
Let $G_1$ be a minimal element with respect to $\prec$ and let $G_2$
be the union of all other essential components.
There is a cell $C$ of $G$, which is bounded by both $G_1$ and $G_2$. Let $C$ correspond
to cell $C_1$ of $G_1$ and cell $C_2$ of $G_2$.
By the definition of $G_1$, $C_1$ is the unbounded cell of $G_1$.
Since $G$ is saturated, at least one of $C_1$ or $C_2$ is a special cell, otherwise $G_1$ and $G_2$ can be connected.

For $i=1, 2$, let $H_i$  be the topological graph $G_i$ together with an isolated vertex in every special cell.
Let $n_i$ denote the number of vertices and $e_i$ the number of edges in $H_i$.
We notice $e=e_1+e_2$ and $n=n_1+n_2-1$ if exactly one of $C_1$ and $C_2$ is a special cell.
Also $n=n_1+n_2-1$ if both of them are special cells, since we can add 1 isolated vertex instead of 2.
By the induction hypothesis, we have $e_i\ge \lfloor2n_i/3\rfloor$,
so $e\ge\lfloor2n_1/3\rfloor+\lfloor2n_2/3\rfloor$,
and it is easy to check, that
for any $n_1, n_2\ge 2$,
$\lfloor2n_1/3\rfloor+\lfloor2n_2/3\rfloor\ge \lfloor2(n_1+n_2-1)/3\rfloor$.
Therefore, 
$e\ge\lfloor2n_1/3\rfloor+\lfloor2n_2/3\rfloor\ge \lfloor2(n_1+n_2-1)/3\rfloor= \lfloor2n/3\rfloor$.
So, if $G$ is not essentially connected, then we reduce the problem and proceed by induction.
\end{proof}

Assume the $3$-simple $2$-plane drawing $G$ has a flag $(u, uv)$.
If $d(v)=1$, then $G$ is isomorphic to $K_2$ and the theorem holds.
If $d(v)=2$, then $G$ contains a $2$-propeller $u,v,w$ by Claim \ref{deg3}. 
Since $G$ is essentially connected, but there is an isolated vertex in every special cell, there is
an isolated vertex $x$ in the special cell of the $2$-propeller.
Therefore, if $d(v)=2$ and $d(w)=1$, then $G$ is isomorphic to a $2$-propeller plus an isolated vertex and we are done. 
If $d(v)=2$ and $d(w)>1$, then remove vertices $u, v, x$. We removed 3 vertices and 2 edges,
so we can use induction.


In the rest of the proof,
we assume that every leaf of $G$ is adjacent to a vertex of degree at least $3$, 
and there is no $2$-propeller subgraph in $G$.
We give weight $3/2$ to every edge.
We discharge the weights to the vertices
and show that either every vertex gets weight at least 1,
or we can prove the lower bound on the number of edges by induction.

Let $uv$ be an edge.
Vertex $u$ gets $1/d(u)$ weight and $v$ gets $1/d(v)$ weight from $uv$. 
Every edge has a non-negative remaining charge.

If $uv$ is a special edge, then it is easy to verify
that $uv$ bounds only one special cell, and the special cell contains an isolated vertex by the assumption, just like in the proof of Claim~\ref{only1}.
In this case, edge $uv$ 
gives the remaining charge to this isolated vertex.
After the discharging step, any vertex $x$ with $d(x)>0$ gets charge at least 1.

Now let $x$ be an isolated vertex, its special cell being $C$. 
We distinguish several cases. 

\smallskip

Case 1: The special cell $C$ has two sides.
Let $u_1v_1$ and $u_2v_2$ be the bounding edges.
They cross twice, in $p$ and $q$ say,
so there are no further crossings on $u_1v_1$ and $u_2v_2$.
The four endpoints are either distinct, 
or two of them $u_1$ and $u_2$ might coincide, if $G$ was $3$-simple.
Suppose the order of crossings on the edges is $u_ipqv_i$, for $i=1,2$.
If the vertices $u_1$ and $u_2$ are distinct, then they can be connected along $u_1p$ and $pu_2$. Therefore, $u_1$ and $u_2$ are either adjacent or coincide in $G$.
Similarly, $v_1$ and $v_2$ are also adjacent.
Therefore, all four endpoints have degree at least $2$, and both $u_1v_1$ and $u_2v_2$ give
at most charge $1/2$ to its endpoints.
Their remaining charges are at least $1/2$, so $x$ gets at least charge $1$.

\smallskip

For the rest of the proof, suppose
$C$ is bounded by
$e_1, e_2, \ldots, e_m$ in clockwise direction,
$e_i=u_iv_i$ such that $\overrightarrow{u_iv_i}$ has $C$ on its right. 

Case 2: $m=3$.
If none of the bounding edges is a flag,
then we are done since each
of those edges give weight at least $1/2$ to $x$.
Suppose that $u_1$ is a leaf.
We can connect $u_1$ to $v_2$ along segments of the edges $u_1v_1$ and $u_2v_2$.
Since $u_1$ is a leaf and the drawing was saturated, $u_1$ and $v_2$ are adjacent, consequently
$v_1=v_2$. Similarly,
we can connect  $u_1$ to $v_3$, so $v_1=v_2=v_3$.

If $u_2$ is not a leaf, then $u_1v_1$ and $u_3v_3$ both give at least $1/6$
to $x$, and $u_2v_2$ gives at least $2/3$, so we have charge at least $1$ for $x$.
The same applies if $u_3$ is not a leaf.
So assume $u_1$, $u_2$ and $u_3$ are all leaves. 
If there are no other edges in $G$, then we can see from the crossing
pattern that
$G$ is a $3$-propeller and an isolated vertex. 
That is, $n=5$ and $e=3$ and the required inequality holds.

Suppose there are further edges. By Claim~\ref{ess}, $G$ is essentially connected. Since
$u_1$, $u_2$, $u_3$ are leaves, $v_1$ is a cut vertex.
Let
$H_1=G\setminus \{x, u_1, u_2, u_3\}$. 
The induced subgraph $H_1$ has $n-4$ vertices and $e-3$ edges, and it is saturated.
Therefore, by the induction hypothesis, $e-3\ge f(n-4)$. Notice that $f(n)\le f(n-4)+3$,
consequently $e\ge f(n)$.

\smallskip

Case 3: $m>3$. 
Each edge gives at least $1/6$ charge to $x$ by Claim~\ref{deg3}. 
If an edge is not a flag, then it gives at least $1/2$ charge to $x$.
If there is at least one non-flag bounding edge, we are done.
Suppose that each edge $u_iv_i$ is a flag
(that is, $d(u_i)$ or $d(v_i)$ is $1$).
We may also assume that $u_1$ is a leaf.
Now, as in the previous case, we can argue that $v_3=v_2=v_1$. 
It implies $u_2$ and $u_3$ are leaves, and by the same argument, $v_5=v_4=v_3=v_2=v_1$.
We can continue and finally we obtain that
all $v_i$ are identical and all $u_i$ are leaves. So the vertices
$u_i, v_i$ $1\le i\le m$ form a star, and they have the same crossing pattern as an $m$-propeller. 
Therefore, $u_i, v_i$ $1\le i\le m$ span an $m$-propeller.
We can finish this case exactly as Case 2. 
If there are no further edges in $G$, then the graph is an $m$-propeller and an isolated vertex. 
That is, $n=m+2$ and $e=m$ and the inequality holds.
If there are further edges,
then $v_1$ is a cut vertex, and we can apply induction. 
This concludes the proof of Theorem~\ref{2-3-simple}.
\end{proof}


\bigskip

\noindent {\bf Remarks.}

\begin{itemize}
  
\item  We have established lower and upper bounds on the number of edges of a saturated, $k$-simple, 2-plane
drawing of a graph. As we mentioned in the introduction, this problem has many modifications, generalizations. Probably the most natural modification is that instead of graphs already drawn, we consider
saturated {\em abstract} graphs.  A graph $G$ is saturated $l$-simple $k$-planar, if 
it has an $l$-simple $k$-planar drawing but adding any edge, 
the resulting graph does not have such a drawing.
Let $t^l_k(n)$ be the minimum number of edges of a saturated $l$-simple $k$-planar graph of $n$ vertices.
By definition, $s^l_k(n)\le t^l_k(n)$. 
We are not aware of any case when the best lower bound on 
 $t^l_k(n)$ is better than for  $s^l_k(n)$.
On the other hand, 
it seems to be much harder to  establish an upper bound construction for 
$t^l_k(n)$ than for  $s^l_k(n)$. In fact, we know nontrivial upper bounds
only in two cases, $t_1^1(n)\le 2.64n+O(1)$
\cite{brandenburg}
and $t_2^1(n)\le 2.63n+O(1)$
\cite{auer}, the latter without a full proof. 

It is known that a $k$-planar graph has at most $c\sqrt{k}n$ edges \cite{pt}, 
so
$t^l_k(n)\le c\sqrt{k}n$, for some $c>0$.

\medskip 

\noindent {\bf Problem 1.} Prove that for every $c>0$,  
$t^l_k(n)\le c\sqrt{k}n$ if $k, l, n$ are large enough.

\medskip

\item For any $n$ and $k$, the best known upper and lower bounds on 
$s_k^l$ decrease or stay the same as we increase $l$.
This would suggest that $s_k^l\le s_k^{l-1}$ for any $n, k, l$, or at least 
if $n$ is large enough, however, we cannot prove it.

\medskip 

\noindent {\bf Problem 2.} 
Is it true, that for any $k$ and $l$, and $n$ large enough,
$s_k^l\le s_k^{l-1}$?

\end{itemize}



\end{document}